\documentclass[pdflatex,sn-mathphys-num]{sn-jnl}

\usepackage{graphicx}%
\usepackage{multirow}%
\usepackage{amsmath,amssymb,amsfonts,mathtools}%
\usepackage{amsthm}%
\usepackage{mathrsfs}%
\usepackage[title]{appendix}%
\usepackage{xcolor}%
\usepackage{textcomp}%
\usepackage{manyfoot}%
\usepackage{booktabs}%
\usepackage{enumitem}%
\usepackage{url}%
\usepackage{hyperref}%

\numberwithin{equation}{section}
\setlist[enumerate]{label=\textup{(\roman*)},leftmargin=2.2em}
\setlist[itemize]{leftmargin=2.0em}
\allowdisplaybreaks

\theoremstyle{thmstyleone}%
\newtheorem{theorem}{Theorem}[section]
\newtheorem{corollary}[theorem]{Corollary}
\newtheorem{proposition}[theorem]{Proposition}
\newtheorem{lemma}[theorem]{Lemma}

\newtheorem{question}[theorem]{Question}

\theoremstyle{thmstyletwo}%
\newtheorem{remark}[theorem]{Remark}
\newtheorem{example}[theorem]{Example}

\theoremstyle{thmstylethree}%
\newtheorem{definition}[theorem]{Definition}

\newcommand{\N}{\mathbb N}
\newcommand{\Nzero}{\mathbb N_0}
\newcommand{\Z}{\mathbb Z}
\newcommand{\Q}{\mathbb Q}
\newcommand{\G}{\mathcal G}
\newcommand{\U}{\mathcal U}
\newcommand{\A}{\mathcal A}
\newcommand{\F}{\mathsf F}
\newcommand{\Zsf}{\mathsf Z}
\newcommand{\Lsf}{\mathsf L}

\newcommand{\eps}{\varepsilon}

\newcommand{\precO}{\preceq_{\Omega}}

\begin{document}

\title[Finite factorization is detected by undermonoids]{Finite factorization is detected by undermonoids}

\author*[1]{\fnm{Yutong} \sur{Zhang}}\email{yutongzhang@stu.scu.edu.cn}
\author[1]{\fnm{Yaoran} \sur{Yang}}\email{yangyaoran@stu.scu.edu.cn}
\affil[1]{\orgdiv{School of Mathematics}, \orgname{Sichuan University}, \orgaddress{\city{Chengdu}, \postcode{610065}, \country{China}}}

\abstract{
Let $M$ be a cancellative commutative monoid and call a submonoid $S$ of $M$ an undermonoid if $\G(S)=\G(M)$ inside the Grothendieck group of $M$.  Gotti and Li asked whether the finite factorization property is hereditary once it is known on all undermonoids: if every undermonoid of $M$ is a finite factorization monoid, must every submonoid of $M$ be a finite factorization monoid?  We give an affirmative answer.  Equivalently, for every cancellative commutative monoid $M$, the following two conditions coincide: every submonoid of $M$ is an FFM, and every undermonoid of $M$ is an FFM.  The proof isolates a fixed length $\ell$ and an infinite set of length-$\ell$ factorizations of one element $b$.  In the non-group case, a divisor-complement ideal $I=\{m\in M:m\nmid_M b\}$ enlarges the bad submonoid to a bad undermonoid while preserving the chosen length-$\ell$ factorizations.  In the group case, a maximality argument over submonoids for which these factorizations survive is combined with a two-sided perturbation $S\mapsto S+\Nzero(2b+u)$.  The key point is that the perturbation creates no new units and does not split any atom occurring in the fixed factorizations.  This yields an undermonoid with infinitely many factorizations of $b$, contradicting the hypothesis.
}

\keywords{cancellative monoid, finite factorization monoid, undermonoid, hereditary factorization, Grothendieck group, Zorn's lemma}

\pacs[Mathematics Subject Classification]{Primary 20M13; Secondary 13A05, 13F15, 20M14}

\maketitle

\section{Introduction}

All monoids in this paper are cancellative and commutative and are written additively.  The identity is denoted by $0$, the group of units by $\U(M)$, and the Grothendieck group by $\G(M)$.  If $N\subseteq M$ is a submonoid, then throughout the paper $\G(N)$ is viewed as the subgroup of $\G(M)$ generated by $N$.  Following Gotti and Li~\cite{GottiLi}, $N$ is called an \emph{undermonoid} of $M$ when
\begin{equation}\label{eq:under-intro}
        \G(N)=\G(M).
\end{equation}
The terminology is intended to measure how much arithmetic of $M$ can be recovered from submonoids that generate the same ambient group.

Let $M_{\mathrm{red}}=M/\U(M)$.  If $\A(M_{\mathrm{red}})$ denotes the set of atoms of the reduced monoid, the factorization monoid of $M$ is
\begin{equation}\label{eq:factorization-monoid-intro}
        \Zsf(M)=\F\big(\A(M_{\mathrm{red}})\big),
\end{equation}
with factorization homomorphism
\begin{equation}\label{eq:factorization-hom-intro}
        \pi_M:\Zsf(M)\longrightarrow M_{\mathrm{red}},
        \qquad
        \pi_M(a)=a\quad(a\in\A(M_{\mathrm{red}})).
\end{equation}
For $b\in M$, put
\begin{equation}\label{eq:factorization-set-intro}
        \Zsf_M(b)=\pi_M^{-1}\big(b+\U(M)\big),
        \qquad
        \Zsf_{M,\ell}(b)=\{z\in\Zsf_M(b): |z|=\ell\}.
\end{equation}
Thus $M$ is a finite factorization monoid, abbreviated FFM, if
\begin{equation}\label{eq:ffm-def-intro}
        1\le |\Zsf_M(b)|<\infty\qquad\text{for every } b\in M.
\end{equation}
The bounded factorization property is the corresponding finiteness condition for lengths,
\begin{equation}\label{eq:bfm-def-intro}
        1\le |\Lsf_M(b)|<\infty,
        \qquad
        \Lsf_M(b)=\{|z|:z\in\Zsf_M(b)\}.
\end{equation}
Consequently
\begin{equation}\label{eq:ffm-decomposition-intro}
       M\text{ is an FFM}
       \iff
       M\text{ is a BFM and } |\Zsf_{M,\ell}(b)|<\infty
       \text{ for all } b,\ell.
\end{equation}

Gotti and Li proved that the bounded factorization property is detected by undermonoids~\cite[Theorem~5.4]{GottiLi}: every undermonoid of $M$ is a BFM if and only if every submonoid of $M$ is a BFM.  They then posed the corresponding finite-factorization problem~\cite[Question~5.5]{GottiLi}.

\begin{question}[Gotti--Li]\label{q:gotti-li}
Let $M$ be a monoid.  If every undermonoid of $M$ is an FFM, does it follow that every submonoid of $M$ is an FFM?
\end{question}

The purpose of this paper is to settle Question~\ref{q:gotti-li}.  We prove the following theorem.

\begin{theorem}[Main theorem]\label{thm:main}
For every cancellative commutative monoid $M$, the following conditions are equivalent.
\begin{enumerate}
\item Every submonoid of $M$ is an FFM.
\item Every undermonoid of $M$ is an FFM.
\end{enumerate}
Equivalently, finite factorization is a hereditary property precisely when it holds on all undermonoids.
\end{theorem}

The implication (i)$\Rightarrow$(ii) is immediate from the definitions.  The proof of the converse has two substantially different parts.  First, because FFMs are BFMs, the Gotti--Li theorem for bounded factorization implies that every submonoid of $M$ is a BFM.  Thus a counterexample to hereditary finite factorization has a single element $b$ and a single length $\ell$ with
\begin{equation}\label{eq:bad-fixed-length-intro}
      |\Zsf_{N,\ell}(b)|=\infty\qquad(\ast)
\end{equation}
inside some submonoid $N\subseteq M$.  The task is to pass from $(N,b,\ell)$ to an undermonoid $W$ in which the same fixed-length infinitude remains visible.

When $M$ is not a group, the construction is direct.  Let
\begin{equation}\label{eq:ideal-intro}
      I_b(M)=\{m\in M:m\nmid_M b\}.
\end{equation}
The set $I_b(M)$ is a nonempty ideal of $M$, and
\begin{equation}\label{eq:non-group-enlargement-intro}
      W=N\cup I_b(M)
\end{equation}
is an undermonoid of $M$.  Since the chosen atoms occurring in $(\ast)$ all divide $b$, none of them becomes decomposable after adjoining $I_b(M)$: any decomposition of such an atom would again consist of divisors of $b$, hence would lie back in $N$.  The infinite family of length-$\ell$ factorizations survives in $W$, contradiction.

When $M$ is a group, the construction is subtler.  We consider submonoids $S$ containing $N$ for which the chosen infinite family of factorizations survives, and then choose a maximal such $S$ under an order that preserves the existing units of smaller submonoids.
If $u\in M$ satisfies either
\begin{equation}\label{eq:group-conditions-intro}
       n_0u\in S\quad\text{for some }n_0\in\N,
       \qquad\text{or}\qquad
       nu+q\ne 0\quad(n\in\N,\ q\in S).\qquad(\dagger)
\end{equation}
then the perturbation
\begin{equation}\label{eq:perturb-intro}
      S'=S+\Nzero(2b+u)
\end{equation}
creates no new units and preserves the chosen atoms.  Maximality therefore forces $2b+u\in S$.  Taking $u=v$ and $u=-v$ for $v\notin\G(S)$ then yields a contradiction, so $\G(S)=M$ and $S$ is the desired bad undermonoid.

The argument is designed to be invariant under units.  This is essential: in the non-reduced case, the fact that two atoms of $N$ remain distinct in a larger monoid is not guaranteed merely by $\U(S)\cap N=\U(N)$.  We therefore work with the stronger unit-reflection condition
\begin{equation}\label{eq:unit-reflection-intro}
       \U(S)\cap\G(N)=\U(N),
\end{equation}
which ensures that the canonical map
\begin{equation}\label{eq:injective-reduction-intro}
       N/\U(N)\longrightarrow S/\U(S)
\end{equation}
is injective.  The maximality argument is carried out inside the class of extensions satisfying \eqref{eq:unit-reflection-intro} and preserving precisely those atom classes that occur in the fixed infinite family \eqref{eq:bad-fixed-length-intro}.

\section{Preliminaries}

We use standard notation from factorization theory in commutative cancellative monoids; see, for example, Geroldinger and Halter-Koch~\cite{GHK}.  This section records the exact conventions needed for the proof, including the unit-sensitive form of factorization persistence used later.  Throughout, $\N=\{1,2,\ldots\}$ and $\Nzero=\N\cup\{0\}$.

\subsection{Divisibility, reductions, and factorizations}

Let $M$ be a monoid.  Since $M$ is cancellative, the canonical map $M\to \G(M)$ is injective; we identify $M$ with its image.  For $a,b\in M$ we write
\begin{equation}\label{eq:divisibility}
        a\mid_M b
        \quad\Longleftrightarrow\quad
        b-a\in M
        \quad\Longleftrightarrow\quad
        b=a+c\text{ for some }c\in M.
\end{equation}
If $a\mid_M b$, then every representative of the class $a+\U(M)$ divides $b$ up to a unit: indeed,
\begin{equation}\label{eq:unit-divisor}
        b=a+c=(a+\eps)+(c-\eps)
        \qquad(\eps\in\U(M)).
\end{equation}
The reduced monoid is
\begin{equation}\label{eq:reduction}
        M_{\mathrm{red}}=M/\U(M),
        \qquad
        \bar x=x+\U(M).
\end{equation}
It is reduced, and $\bar x$ is an atom of $M_{\mathrm{red}}$ precisely when $x\notin\U(M)$ and every equality
\begin{equation}\label{eq:atom-class}
        x=y+z+\eps,
        \qquad y,z\in M,
        \quad\eps\in\U(M),
\end{equation}
forces $y\in\U(M)$ or $z\in\U(M)$.  We write
\begin{equation}\label{eq:Z-M}
        \Zsf(M)=\F(\A(M_{\mathrm{red}}))
\end{equation}
for the free commutative monoid on the atom classes of $M_{\mathrm{red}}$.  Its elements are written additively,
\begin{equation}\label{eq:z-expression}
        z=\bar a_1+\cdots+\bar a_\ell,
        \qquad \bar a_i\in\A(M_{\mathrm{red}}),
        \qquad |z|=\ell.
\end{equation}
The factorization homomorphism is
\begin{equation}\label{eq:phi-M}
        \varphi_M:\Zsf(M)\longrightarrow M_{\mathrm{red}},
        \qquad
        \varphi_M(\bar a)=\bar a.
\end{equation}
For $b\in M$,
\begin{align}
        \Zsf_M(b)
        &=\{z\in\Zsf(M):\varphi_M(z)=\bar b\},\label{eq:Z-b}\\
        \Zsf_{M,\ell}(b)
        &=\{z\in\Zsf_M(b):|z|=\ell\},\label{eq:Z-b-ell}\\
        \Lsf_M(b)
        &=\{|z|:z\in\Zsf_M(b)\}.\label{eq:L-b}
\end{align}
If $z=\bar a_1+\cdots+\bar a_\ell\in\Zsf_M(b)$ and $a_i\in M$ are chosen representatives, then there exists $\eps\in\U(M)$ such that
\begin{equation}\label{eq:lifting-factorization}
        b=a_1+\cdots+a_\ell+\eps.
\end{equation}
Conversely, if $\bar a_i\in\A(M_{\mathrm{red}})$ and \eqref{eq:lifting-factorization} holds for some unit $\eps$, then $\bar a_1+\cdots+\bar a_\ell\in\Zsf_M(b)$.

\begin{definition}\label{def:BFM-FFM-LFFM}
A monoid $M$ is
\begin{enumerate}
\item atomic if $\Zsf_M(b)\ne\varnothing$ for every $b\in M$;
\item a BFM if $1\le |\Lsf_M(b)|<\infty$ for every $b\in M$;
\item an LFFM if $M$ is atomic and $|\Zsf_{M,\ell}(b)|<\infty$ for all $b\in M$ and $\ell\in\N$;
\item an FFM if $1\le |\Zsf_M(b)|<\infty$ for every $b\in M$.
\end{enumerate}
\end{definition}

The following elementary criterion is used only after hereditary bounded factorization has already been obtained.

\begin{lemma}\label{lem:bfm-lffm-ffm}
For a monoid $M$,
\begin{equation}\label{eq:ffm-iff-bfm-lffm}
       M\text{ is an FFM}
       \quad\Longleftrightarrow\quad
       M\text{ is a BFM and an LFFM}.
\end{equation}
Equivalently, if $M$ is a BFM and is not an FFM, then there exist $b\in M$, $\ell\in\N$, and an infinite set
\begin{equation}\label{eq:infinite-slice}
       \Omega\subseteq\Zsf_{M,\ell}(b),
       \qquad |\Omega|=\infty.
\end{equation}
Moreover, $b\notin\U(M)$.
\end{lemma}

\begin{proof}
If $M$ is an FFM, then every $\Zsf_M(b)$ is nonempty and finite, so both the set of lengths and every fixed-length fiber are finite.  Conversely, if $M$ is a BFM and an LFFM, then
\begin{equation}\label{eq:finite-union}
       \Zsf_M(b)=\bigsqcup_{\ell\in\Lsf_M(b)}\Zsf_{M,\ell}(b)
\end{equation}
for every $b\in M$.  The right side is a finite disjoint union of finite sets; hence $M$ is an FFM.  The final assertion follows by taking $b$ with $\Zsf_M(b)$ infinite.  Since $\Lsf_M(b)$ is finite, some fiber $\Zsf_{M,\ell}(b)$ is infinite.  If $b\in\U(M)$, then $\Zsf_M(b)=\{\emptyset\}$, contradiction.
\end{proof}

\subsection{Unit-reflecting extensions}

Let $N\subseteq S$ be submonoids of a common monoid $M$.  The inclusion $N\hookrightarrow S$ induces a homomorphism on reductions
\begin{equation}\label{eq:reduction-map}
       \rho_{N,S}:N/\U(N)\longrightarrow S/\U(S),
       \qquad
       n+\U(N)\longmapsto n+\U(S).
\end{equation}
The map $\rho_{N,S}$ need not be injective.  Injectivity is precisely the condition that $S$ introduces no new unit relation among elements of $\G(N)$.

\begin{lemma}\label{lem:unit-reflecting}
Let $N\subseteq S\subseteq M$ be submonoids.  The following conditions are equivalent.
\begin{enumerate}
\item The map $\rho_{N,S}$ in \eqref{eq:reduction-map} is injective.
\item $\U(S)\cap\G(N)=\U(N)$.
\item Whenever $x,y\in N$ and $x-y\in\U(S)$ in $\G(M)$, one has $x-y\in\U(N)$.
\end{enumerate}
\end{lemma}

\begin{proof}
The equality $\rho_{N,S}(x+\U(N))=\rho_{N,S}(y+\U(N))$ is equivalent to
\begin{equation}\label{eq:rho-equality}
       x-y\in \U(S)\cap\G(N).
\end{equation}
Thus $\rho_{N,S}$ is injective if and only if every element of the intersection in \eqref{eq:rho-equality} already belongs to $\U(N)$.  This is exactly (ii), and (ii) is plainly equivalent to (iii).
\end{proof}

\begin{definition}\label{def:unit-reflecting}
An inclusion $N\subseteq S$ is called \emph{unit-reflecting over $N$} if
\begin{equation}\label{eq:unit-reflecting}
       \U(S)\cap\G(N)=\U(N).
\end{equation}
\end{definition}

Fix a submonoid $N$ and an element $b\in N$.  Suppose
\begin{equation}\label{eq:fixed-Omega}
       \Omega\subseteq\Zsf_{N,\ell}(b)
\end{equation}
for some $\ell\in\N$.  Define $\mathcal C_N(\Omega)$ to be the set of atom classes of $N_{\mathrm{red}}$ occurring in at least one factorization from $\Omega$:
\begin{equation}\label{eq:C-Omega}
       \mathcal C_N(\Omega)
       =\left\{
       \alpha\in\A(N_{\mathrm{red}}):
       z(\alpha)>0\text{ for some }z\in\Omega
       \right\}.
\end{equation}
For each $\alpha\in\mathcal C_N(\Omega)$ choose once and for all a lift
\begin{equation}\label{eq:chosen-lift}
       a_\alpha\in N,
       \qquad
       a_\alpha+\U(N)=\alpha.
\end{equation}
Since $\alpha$ occurs in a factorization of $b$, there exist
\begin{equation}\label{eq:atom-divides-b}
       r_\alpha\in N,
       \qquad
       \eps_\alpha\in\U(N)
\end{equation}
such that
\begin{equation}\label{eq:a-alpha-divides-b}
       b=a_\alpha+r_\alpha+\eps_\alpha.
\end{equation}
Here $r_\alpha$ is the sum of representative atoms for the remaining terms of one chosen factorization containing $\alpha$.

\begin{definition}\label{def:Omega-admissible}
Let $N\subseteq M$, $b\in N$, and $\Omega\subseteq\Zsf_{N,\ell}(b)$.  A submonoid $S$ with $N\subseteq S\subseteq M$ is called \emph{$\Omega$-admissible over $N$} if
\begin{align}
       \U(S)\cap\G(N)&=\U(N),\label{eq:admissible-unit}\\
       a_\alpha+\U(S)&\in\A(S_{\mathrm{red}})
          \quad\text{for every }\alpha\in\mathcal C_N(\Omega).
          \label{eq:admissible-atoms}
\end{align}
\end{definition}

The next lemma records the reason for this definition.

\begin{lemma}[persistence of a fixed family]\label{lem:persistence}
Let $N\subseteq S\subseteq M$ be $\Omega$-admissible over $N$, with $\Omega\subseteq\Zsf_{N,\ell}(b)$.  Then there is an injective map
\begin{equation}\label{eq:persistence-map}
        \Theta_{N,S}:\Omega\hookrightarrow\Zsf_{S,\ell}(b)
\end{equation}
obtained by applying $\rho_{N,S}$ to each atom class in each factorization.  In particular, if $\Omega$ is infinite, then $S$ is not an FFM.
\end{lemma}

\begin{proof}
Let
\begin{equation}\label{eq:z-in-Omega}
      z=\alpha_1+\cdots+\alpha_\ell\in\Omega,
      \qquad \alpha_i\in\mathcal C_N(\Omega).
\end{equation}
Define
\begin{equation}\label{eq:Theta-def}
      \Theta_{N,S}(z)=
      (a_{\alpha_1}+\U(S))+\cdots+(a_{\alpha_\ell}+\U(S)).
\end{equation}
Condition \eqref{eq:admissible-atoms} ensures that every summand in \eqref{eq:Theta-def} is an atom of $S_{\mathrm{red}}$.  Since $z\in\Zsf_N(b)$, there exists $\eps\in\U(N)\subseteq\U(S)$ such that
\begin{equation}\label{eq:z-lift-b}
      b=a_{\alpha_1}+\cdots+a_{\alpha_\ell}+\eps.
\end{equation}
Hence $\Theta_{N,S}(z)\in\Zsf_{S,\ell}(b)$.

It remains to prove injectivity.  Suppose
\begin{equation}\label{eq:Theta-equal}
      \Theta_{N,S}(z)=\Theta_{N,S}(z')
\end{equation}
for
\begin{equation}\label{eq:z-zprime}
      z=\sum_{\alpha\in\mathcal C_N(\Omega)}m_\alpha\alpha,
      \qquad
      z'=\sum_{\alpha\in\mathcal C_N(\Omega)}m'_\alpha\alpha.
\end{equation}
Equality in the free monoid $\Zsf(S)$ means that the multisets of atom classes
\begin{equation}\label{eq:multisets-S}
      \{a_\alpha+\U(S)\text{ with multiplicity }m_\alpha
      :\alpha\in\mathcal C_N(\Omega)
      \}
\end{equation}
for $z$ and $z'$ coincide.  If
\begin{equation}\label{eq:classes-equal-S}
      a_\alpha+\U(S)=a_\beta+\U(S),
\end{equation}
then $a_\alpha-a_\beta\in\U(S)\cap\G(N)=\U(N)$ by \eqref{eq:admissible-unit}.  Hence
\begin{equation}\label{eq:classes-equal-N}
      a_\alpha+\U(N)=a_\beta+\U(N),
\end{equation}
so $\alpha=\beta$.  Thus the multisets in $S_{\mathrm{red}}$ agree only when the original multisets in $N_{\mathrm{red}}$ agree.  Therefore $z=z'$.
\end{proof}

\section{The non-group enlargement}

This section treats the case in which the ambient monoid is not a group.  The construction is an ideal enlargement: one adjoins every element that fails to divide the chosen element $b$.  The point is that any decomposition of an atom dividing $b$ must again use only divisors of $b$, and hence cannot involve the adjoined ideal.

\begin{lemma}[the divisor-complement ideal]\label{lem:divisor-complement}
Let $M$ be a monoid that is not a group and let $b\in M$.  Put
\begin{equation}\label{eq:I-b}
      I=I_b(M):=\{m\in M:m\nmid_M b\}.
\end{equation}
Then:
\begin{enumerate}
\item $I$ is a nonempty ideal of $M$;
\item $I\cap\U(M)=\varnothing$;
\item if $N\subseteq M$ is any submonoid containing $b$, then
      \begin{equation}\label{eq:M-prime}
          M'=N\cup I
      \end{equation}
      is a submonoid of $M$;
\item $\U(M')=\U(N)$;
\item $M'$ is an undermonoid of $M$.
\end{enumerate}
\end{lemma}

\begin{proof}
Since $M$ is not a group, choose $x\in M\setminus\U(M)$.  Then $b+x\in I$: if $b+x\mid_M b$, then
\begin{equation}\label{eq:b-plus-x-divides}
      b=b+x+t
\end{equation}
for some $t\in M$, whence $x+t=0$ by cancellation in $\G(M)$, contradicting $x\notin\U(M)$.  Thus $I\ne\varnothing$.

If $m\in I$ and $q\in M$, then $m+q\in I$.  Indeed, if $m+q\mid_M b$, then
\begin{equation}\label{eq:ideal-proof}
      b=m+q+t=m+(q+t)
\end{equation}
for some $t\in M$, so $m\mid_M b$, contradiction.  Therefore $I+M\subseteq I$, and $I$ is an ideal.  If $u\in\U(M)$, then
\begin{equation}\label{eq:unit-divides-b}
      b=u+(b-u),
      \qquad b-u\in M,
\end{equation}
so $u\mid_M b$ and $u\notin I$.  This proves (ii).

Now let $N\subseteq M$ be a submonoid containing $b$.  Since $I$ is an ideal, sums involving at least one element of $I$ belong to $I$, while sums of two elements of $N$ belong to $N$; hence $M'=N\cup I$ is a submonoid.  Because $I$ contains no unit of $M$, it contains no unit of $M'$.  Let $u\in\U(M')$.  Then $u\in\U(M)$, so $u\notin I$ and hence $u\in N$.  If $v\in M'$ satisfies $u+v=0$, then $v\in\U(M)$, whence $v\notin I$ and $v\in N$.  Thus $u\in\U(N)$.  The reverse inclusion is immediate from $N\subseteq M'$.  Therefore
\begin{equation}\label{eq:units-M-prime}
      \U(M')=\U(N).
\end{equation}

Finally choose $c\in I$.  For every $m\in M$, the ideal property gives
\begin{equation}\label{eq:c-plus-m-in-I}
      c+m\in I\subseteq M'.
\end{equation}
Since also $c\in M'$, we have
\begin{equation}\label{eq:m-in-G-M-prime}
      m=(c+m)-c\in\G(M').
\end{equation}
Hence $\G(M)\subseteq\G(M')$.  The reverse inclusion is automatic because $M'\subseteq M$.  Therefore $\G(M')=\G(M)$.
\end{proof}

\begin{lemma}[preservation of fixed atoms in the non-group case]\label{lem:non-group-preserve-atoms}
Let $M$ be a non-group monoid, let $N\subseteq M$ be a submonoid, let $b\in N$, and let
\begin{equation}\label{eq:Omega-non-group}
      \Omega\subseteq\Zsf_{N,\ell}(b).
\end{equation}
Set $M'=N\cup I_b(M)$.  Then $M'$ is $\Omega$-admissible over $N$.
\end{lemma}

\begin{proof}
By Lemma~\ref{lem:divisor-complement}, $\U(M')=\U(N)$, so
\begin{equation}\label{eq:unit-reflection-non-group}
      \U(M')\cap\G(N)=\U(N).
\end{equation}
It remains to verify atom preservation.  Fix $\alpha\in\mathcal C_N(\Omega)$ and let $a=a_\alpha$ be the chosen representative.  By \eqref{eq:a-alpha-divides-b},
\begin{equation}\label{eq:a-divides-b-non-group}
      b=a+r+\eps
\end{equation}
for some $r\in N$ and $\eps\in\U(N)=\U(M')$.  Also $a\notin\U(M')$, since otherwise $a\in\U(N)$ by Lemma~\ref{lem:divisor-complement}, contradicting that $a+\U(N)=\alpha$ is an atom.

Assume that $a+\U(M')$ decomposes in $M'_{\mathrm{red}}$.  Then there exist $x,y\in M'$ and $\eta\in\U(M')$ such that
\begin{equation}\label{eq:a-decompose-M-prime}
      a=x+y+\eta.
\end{equation}
Combining \eqref{eq:a-divides-b-non-group} and \eqref{eq:a-decompose-M-prime} gives
\begin{equation}\label{eq:b-x-y}
      b=x+y+r+\eps+\eta.
\end{equation}
Hence $x\mid_M b$ and $y\mid_M b$.  Therefore $x,y\notin I_b(M)$, and since $M'=N\cup I_b(M)$ we get
\begin{equation}\label{eq:x-y-in-N}
      x,y\in N.
\end{equation}
Now \eqref{eq:a-decompose-M-prime} is a decomposition of $a+\U(N)$ in $N_{\mathrm{red}}$.  Since $\alpha=a+\U(N)$ is an atom of $N_{\mathrm{red}}$, either
\begin{equation}\label{eq:atom-forces-unit-non-group}
      x\in\U(N)=\U(M')
      \qquad\text{or}\qquad
      y\in\U(N)=\U(M').
\end{equation}
Thus $a+\U(M')$ is an atom of $M'_{\mathrm{red}}$.
\end{proof}

\begin{proposition}[bad submonoids produce bad undermonoids: non-group case]\label{prop:non-group-bad-under}
Let $M$ be a monoid that is not a group.  Suppose $N\subseteq M$ is a submonoid for which there exist $b\in N$, $\ell\in\N$, and an infinite set
\begin{equation}\label{eq:bad-non-group}
      \Omega\subseteq\Zsf_{N,\ell}(b).
\end{equation}
Then $M$ has an undermonoid that is not an FFM.
\end{proposition}

\begin{proof}
Set
\begin{equation}\label{eq:W-non-group}
      W=N\cup I_b(M).
\end{equation}
By Lemma~\ref{lem:divisor-complement}, $W$ is an undermonoid of $M$.  By Lemma~\ref{lem:non-group-preserve-atoms}, $W$ is $\Omega$-admissible over $N$.  Lemma~\ref{lem:persistence} gives an injection
\begin{equation}\label{eq:Omega-injects-W}
      \Omega\hookrightarrow \Zsf_{W,\ell}(b).
\end{equation}
Since $\Omega$ is infinite, $\Zsf_W(b)$ is infinite, and $W$ is not an FFM.
\end{proof}

\section{The group case: maximal survival of factorizations}

We now assume that the ambient monoid $M$ is a group.  Then $\G(M)=M$, and a submonoid $S\subseteq M$ is an undermonoid precisely when
\begin{equation}\label{eq:group-undermonoid}
       \G(S)=M.
\end{equation}
The proof proceeds by fixing an infinite family of factorizations and enlarging the initial bad submonoid as much as possible without destroying that family.

\subsection{The survival poset}

Let $M$ be a group, let $N\subseteq M$ be a submonoid, let $b\in N$, and fix an infinite family
\begin{equation}\label{eq:Omega-group}
      \Omega\subseteq\Zsf_{N,\ell}(b).
\end{equation}
We write $\mathscr S_\Omega(N,M)$ for the set of all $\Omega$-admissible submonoids $S$ satisfying
\begin{equation}\label{eq:S-between}
      N\subseteq S\subseteq M.
\end{equation}
The set is nonempty because $N\in\mathscr S_\Omega(N,M)$.

Define a relation $\precO$ on $\mathscr S_\Omega(N,M)$ by
\begin{equation}\label{eq:order-Omega}
      S_1\precO S_2
      \quad\Longleftrightarrow\quad
      S_1\subseteq S_2
      \quad\text{and}\quad
      \U(S_1)=S_1\cap\U(S_2).
\end{equation}
The second condition says only that no element of $S_1$ becomes a new unit after passing to $S_2$.  It is weaker than unit-reflection over $S_1$, which would require
\begin{equation}\label{eq:strong-unit-reflection-over-S1}
      \U(S_2)\cap\G(S_1)=\U(S_1).
\end{equation}
This weaker condition is sufficient for the chain-union and maximality arguments below.  The strong unit-reflection needed to keep the fixed factorizations injective is imposed separately in Definition~\ref{def:Omega-admissible}, namely through the condition
\begin{equation}\label{eq:strong-unit-reflection-over-N-reminder}
      \U(S)\cap\G(N)=\U(N).
\end{equation}

\begin{lemma}\label{lem:survival-poset}
The relation $\precO$ makes $\mathscr S_\Omega(N,M)$ into a partially ordered set.
\end{lemma}

\begin{proof}
Reflexivity is immediate.  If $S_1\precO S_2$ and $S_2\precO S_1$, then $S_1=S_2$ from the inclusions.  For transitivity, suppose
\begin{equation}\label{eq:transitivity-assumption}
      S_1\precO S_2\precO S_3.
\end{equation}
Then $S_1\subseteq S_3$.  Also,
\begin{align}
      S_1\cap\U(S_3)
      &=S_1\cap(S_2\cap\U(S_3)) &&\text{because }S_1\subseteq S_2
          \label{eq:transitivity-1}\\
      &=S_1\cap\U(S_2) &&\text{because }S_2\precO S_3
          \label{eq:transitivity-2}\\
      &=\U(S_1) &&\text{because }S_1\precO S_2.
          \label{eq:transitivity-3}
\end{align}
Therefore $S_1\precO S_3$.
\end{proof}

\begin{lemma}[chain unions]\label{lem:chain-union}
Let $(S_i)_{i\in\Lambda}$ be a nonempty chain in $(\mathscr S_\Omega(N,M),\precO)$.  Put
\begin{equation}\label{eq:chain-union}
      S=\bigcup_{i\in\Lambda}S_i.
\end{equation}
Then $S\in\mathscr S_\Omega(N,M)$ and $S_i\precO S$ for every $i\in\Lambda$.
\end{lemma}

\begin{proof}
Since the $S_i$ form a chain by inclusion and each $S_i$ is a submonoid, $S$ is a submonoid of $M$ containing $N$.

Fix $i\in\Lambda$.  We first prove
\begin{equation}\label{eq:unit-chain}
      \U(S_i)=S_i\cap\U(S).
\end{equation}
The inclusion $\subseteq$ is clear.  Conversely, let $u\in S_i\cap\U(S)$.  Then $-u\in S$, so $-u\in S_j$ for some $j\in\Lambda$.  Because the family is a chain, either $S_i\subseteq S_j$ or $S_j\subseteq S_i$.  If $S_i\subseteq S_j$, then
\begin{equation}\label{eq:u-unit-if-Si-sub-Sj}
      u\in S_i\cap\U(S_j)=\U(S_i)
\end{equation}
by $S_i\precO S_j$.  If $S_j\subseteq S_i$, then $u,-u\in S_i$, so $u\in\U(S_i)$.  This proves \eqref{eq:unit-chain}.

Now $S$ is unit-reflecting over $N$.  Indeed, if
\begin{equation}\label{eq:u-in-US-GN}
      u\in\U(S)\cap\G(N),
\end{equation}
then $u\in S_i$ and $-u\in S_j$ for suitable $i,j$.  Choose $k$ with $S_i\cup S_j\subseteq S_k$; then $u\in\U(S_k)\cap\G(N)=\U(N)$.  Thus
\begin{equation}\label{eq:S-unit-reflecting}
      \U(S)\cap\G(N)=\U(N).
\end{equation}

It remains to verify atom preservation.  Fix $\alpha\in\mathcal C_N(\Omega)$ and write $a=a_\alpha$.  First $a\notin\U(S)$: otherwise $-a\in S_j$ for some $j\in\Lambda$, while $a\in N\subseteq S_j$, so $a\in\U(S_j)$, contradicting that $a+\U(S_j)$ is an atom.  Now suppose
\begin{equation}\label{eq:a-decomposes-union}
      a=x+y+\eps,
      \qquad x,y\in S,
      \quad \eps\in\U(S).
\end{equation}
Choose $i,j,k\in\Lambda$ such that $x\in S_i$, $y\in S_j$, and $\eps,-\eps\in S_k$.  Since the set is a chain, there exists $h\in\Lambda$ with
\begin{equation}\label{eq:h-contains}
      S_i\cup S_j\cup S_k\subseteq S_h.
\end{equation}
Then \eqref{eq:a-decomposes-union} is a decomposition of $a+\U(S_h)$ in $(S_h)_{\mathrm{red}}$.  Since $S_h$ is $\Omega$-admissible, $a+\U(S_h)$ is an atom.  Hence
\begin{equation}\label{eq:x-or-y-unit-Sh}
      x\in\U(S_h)
      \quad\text{or}\quad
      y\in\U(S_h).
\end{equation}
By \eqref{eq:unit-chain}, $\U(S_h)=S_h\cap\U(S)$, so $x\in\U(S)$ or $y\in\U(S)$.  Therefore $a+\U(S)$ is an atom of $S_{\mathrm{red}}$.  Thus $S\in\mathscr S_\Omega(N,M)$, and \eqref{eq:unit-chain} gives $S_i\precO S$ for all $i$.
\end{proof}

\begin{proposition}[maximal survival]\label{prop:maximal-survival}
The poset $(\mathscr S_\Omega(N,M),\precO)$ has a maximal element.
\end{proposition}

\begin{proof}
Every nonempty chain has an upper bound by Lemma~\ref{lem:chain-union}, and the empty chain has an upper bound because $\mathscr S_\Omega(N,M)$ is nonempty.  Zorn's lemma applies.
\end{proof}

\subsection{The perturbation lemma}

Let $S\in\mathscr S_\Omega(N,M)$ and $u\in M$.  Put
\begin{equation}\label{eq:w-def}
      w=2b+u,
      \qquad
      S'=S+\Nzero w=\bigcup_{n\ge 0}(S+nw).
\end{equation}
The next lemma is the technical heart of the proof.  It shows that, under the standard alternative used in the Gotti--Li maximality method, the enlargement $S\subseteq S'$ does not alter units and does not split any atom used by $\Omega$.

\begin{lemma}[no new units]\label{lem:no-new-units}
Let $M$ be a group, let $S\subseteq M$ be a submonoid, and let $b\in S\setminus\U(S)$.  Let $u\in M$ satisfy at least one of the following conditions:
\begin{align}
&\exists n_0\in\N\text{ such that }n_0u\in S, \label{eq:condition-i}\\
&nu+q\ne0\quad\text{for all }n\in\N\text{ and all }q\in S. \label{eq:condition-ii}
\end{align}
Set $w=2b+u$ and $S'=S+\Nzero w$.  Then
\begin{equation}\label{eq:units-Sprime-S}
      \U(S')=\U(S).
\end{equation}
In particular, $w\notin\U(S')$ unless $w\in\U(S)$, and the latter cannot occur under \eqref{eq:condition-i} or \eqref{eq:condition-ii}.
\end{lemma}

\begin{proof}
It suffices to show that an element of the form $s+kw$ with $s\in S$ and $k\in\Nzero$ is invertible in $S'$ only when $k=0$ and $s\in\U(S)$.

Assume that
\begin{equation}\label{eq:invertible-skw}
      s+kw\in\U(S').
\end{equation}
Then there exist $t\in S$ and $m\in\Nzero$ such that
\begin{equation}\label{eq:inverse-skw}
      s+kw+t+mw=0.
\end{equation}
Put $K=k+m$.  If $K=0$, then $k=m=0$ and $s+t=0$, so $s\in\U(S)$.

Suppose $K\ge1$.  Since $w=2b+u$, equation \eqref{eq:inverse-skw} becomes
\begin{equation}\label{eq:K-u-equation}
      Ku+(s+t+2Kb)=0.
\end{equation}
If \eqref{eq:condition-ii} holds, this contradicts \eqref{eq:condition-ii} with $n=K$ and $q=s+t+2Kb\in S$.

Assume instead that \eqref{eq:condition-i} holds, so $n_0u\in S$ for some $n_0\in\N$.  Multiplying \eqref{eq:K-u-equation} by $n_0$ gives
\begin{equation}\label{eq:n0-multiple}
      K(n_0u)+n_0s+n_0t+2Kn_0b=0.
\end{equation}
Thus
\begin{equation}\label{eq:b-unit-contradiction}
      b+\big(K(n_0u)+n_0s+n_0t+(2Kn_0-1)b\big)=0.
\end{equation}
The second summand lies in $S$, because $K\ge1$ and $2Kn_0-1\ge1$.  Hence $b\in\U(S)$, contradicting the hypothesis.  Therefore $K=0$ is the only possible case, proving \eqref{eq:units-Sprime-S}.

It remains only to justify the final assertion.  If $w\in\U(S')$, then \eqref{eq:units-Sprime-S} gives $w\in\U(S)$.  This cannot happen under \eqref{eq:condition-ii}: indeed, $-w\in S$ and then, with $q=2b-w\in S$, one has $u+q=0$, contradicting \eqref{eq:condition-ii}.  It cannot happen under \eqref{eq:condition-i} either.  If $n_0u\in S$ and $w\in\U(S)$, then $-n_0w\in S$ and
\begin{equation}\label{eq:w-unit-forces-b-unit}
      b+\big(n_0u-n_0w+(2n_0-1)b\big)=0,
\end{equation}
where the second summand belongs to $S$ because $n_0\in\N$.  This makes $b$ a unit of $S$, again contradicting the hypothesis.
\end{proof}

\begin{lemma}[atoms used by $\Omega$ do not split]\label{lem:atoms-do-not-split}
Let $M$ be a group, let $N\subseteq S\subseteq M$, and suppose that $S$ is $\Omega$-admissible over $N$ for some infinite set
\begin{equation}\label{eq:Omega-split}
      \Omega\subseteq\Zsf_{N,\ell}(b).
\end{equation}
Assume $b\notin\U(S)$.  Let $u\in M$ satisfy \eqref{eq:condition-i} or \eqref{eq:condition-ii}, and set
\begin{equation}\label{eq:Sprime-split}
      w=2b+u,
      \qquad
      S'=S+\Nzero w.
\end{equation}
Then $S'$ is $\Omega$-admissible over $N$.
\end{lemma}

\begin{proof}
By Lemma~\ref{lem:no-new-units},
\begin{equation}\label{eq:units-split}
      \U(S')=\U(S).
\end{equation}
Since $S$ is $\Omega$-admissible over $N$, we get
\begin{equation}\label{eq:unit-reflection-Sprime}
      \U(S')\cap\G(N)=\U(S)\cap\G(N)=\U(N).
\end{equation}
It remains to prove atom preservation.

Fix $\alpha\in\mathcal C_N(\Omega)$ and put $a=a_\alpha$.  Since $\alpha$ occurs in a factorization of $b$, choose $r=r_\alpha\in N$ and $\eps=\eps_\alpha\in\U(N)$ with
\begin{equation}\label{eq:b-a-r-eps-split}
      b=a+r+\eps.
\end{equation}
Because $S$ is $\Omega$-admissible, $a+\U(S)$ is an atom of $S_{\mathrm{red}}$.  In particular, $a\notin\U(S)=\U(S')$.

Suppose that $a+\U(S')$ decomposes in $S'_{\mathrm{red}}$.  Then there are $x,y\in S'$ and $\eta\in\U(S')$ such that
\begin{equation}\label{eq:a-x-y-eta}
      a=x+y+\eta.
\end{equation}
Write
\begin{equation}\label{eq:x-y-skw}
      x=s_1+k_1w,
      \qquad
      y=s_2+k_2w,
      \qquad
      s_1,s_2\in S,
      \quad k_1,k_2\in\Nzero,
\end{equation}
and put $K=k_1+k_2$.  Combining \eqref{eq:b-a-r-eps-split}, \eqref{eq:a-x-y-eta}, and \eqref{eq:x-y-skw}, and using $w=2b+u$, yields
\begin{align}
      b
      &=a+r+\eps \notag\\
      &=s_1+s_2+Kw+\eta+r+\eps \notag\\
      &=s_1+s_2+r+\eta+\eps+2Kb+Ku.
      \label{eq:b-expanded-split}
\end{align}
Rearranging in the ambient group gives
\begin{equation}\label{eq:Ku-plus-q}
      Ku+\big(s_1+s_2+r+\eta+\eps+(2K-1)b\big)=0.
\end{equation}
If $K\ge1$ and \eqref{eq:condition-ii} holds, this contradicts \eqref{eq:condition-ii}, because
\begin{equation}\label{eq:q-in-S}
      s_1+s_2+r+\eta+\eps+(2K-1)b\in S.
\end{equation}
If $K\ge1$ and \eqref{eq:condition-i} holds, choose $n_0\in\N$ with $n_0u\in S$ and multiply \eqref{eq:Ku-plus-q} by $n_0$.  We get
\begin{equation}\label{eq:b-unit-split}
      b+\big(Kn_0u+n_0s_1+n_0s_2+n_0r+n_0\eta+n_0\eps+\big(n_0(2K-1)-1\big)b\big)=0.
\end{equation}
The second summand lies in $S$, because $K\ge1$ and $n_0(2K-1)-1\ge0$.  This makes $b$ a unit of $S$, contradiction.  Thus
\begin{equation}\label{eq:K-zero}
      K=0.
\end{equation}
Since $k_1,k_2\in\Nzero$, equation \eqref{eq:K-zero} gives $k_1=k_2=0$, and hence $x=s_1,y=s_2\in S$.  Also $\eta\in\U(S')=\U(S)$.  Therefore \eqref{eq:a-x-y-eta} is a decomposition of the atom $a+\U(S)$ in $S_{\mathrm{red}}$, so either
\begin{equation}\label{eq:x-or-y-unit-S}
      x\in\U(S)=\U(S')
      \qquad\text{or}\qquad
      y\in\U(S)=\U(S').
\end{equation}
This proves that $a+\U(S')$ is an atom of $S'_{\mathrm{red}}$.
\end{proof}

\begin{proposition}[maximality forces $2b+u$]\label{prop:maximal-forces}
Let $M$ be a group, let $N\subseteq M$, let $b\in N$, and let
\begin{equation}\label{eq:Omega-prop-max}
      \Omega\subseteq\Zsf_{N,\ell}(b)
\end{equation}
be infinite.  Let $S$ be a maximal element of $\mathscr S_\Omega(N,M)$ under $\precO$.  If $u\in M$ satisfies \eqref{eq:condition-i} or \eqref{eq:condition-ii}, then
\begin{equation}\label{eq:2b-u-in-S}
      2b+u\in S.
\end{equation}
\end{proposition}

\begin{proof}
Because $S$ is $\Omega$-admissible, Lemma~\ref{lem:persistence} gives an injection
\begin{equation}\label{eq:Omega-in-S}
      \Omega\hookrightarrow\Zsf_{S,\ell}(b).
\end{equation}
Thus $\Zsf_S(b)$ is infinite.  In particular, $b\notin\U(S)$.

Set $w=2b+u$ and $S'=S+\Nzero w$.  By Lemma~\ref{lem:no-new-units},
\begin{equation}\label{eq:S-Sprime-units}
      \U(S')=\U(S),
\end{equation}
so $S\precO S'$.  By Lemma~\ref{lem:atoms-do-not-split}, $S'$ is $\Omega$-admissible over $N$, hence
\begin{equation}\label{eq:Sprime-in-poset}
      S'\in\mathscr S_\Omega(N,M).
\end{equation}
The maximality of $S$ in the poset gives $S'=S$.  Therefore $w=2b+u\in S$.
\end{proof}

\subsection{Forcing the generated group to be all of $M$}

We now use Proposition~\ref{prop:maximal-forces} twice, with $u=v$ and $u=-v$, to prove that a maximal survival submonoid must generate the whole ambient group.

\begin{lemma}\label{lem:maximal-is-under}
Let $M$ be a group, let $N\subseteq M$, let $b\in N$, and let
\begin{equation}\label{eq:Omega-max-under}
      \Omega\subseteq\Zsf_{N,\ell}(b)
\end{equation}
be infinite.  If $S$ is a maximal element of $\mathscr S_\Omega(N,M)$, then
\begin{equation}\label{eq:G-S-M}
      \G(S)=M.
\end{equation}
Consequently $S$ is an undermonoid of $M$.
\end{lemma}

\begin{proof}
Assume, to the contrary, that
\begin{equation}\label{eq:G-S-proper}
      \G(S)\subsetneq M.
\end{equation}
Choose
\begin{equation}\label{eq:v-outside-GS}
      v\in M\setminus\G(S).
\end{equation}
Then neither $2b+v$ nor $2b-v$ belongs to $S$, because $b\in S$ and either inclusion would imply
\begin{equation}\label{eq:v-in-GS-if}
      v=(2b+v)-2b\in\G(S)
      \qquad\text{or}\qquad
      v=2b-(2b-v)\in\G(S),
\end{equation}
respectively.

Apply Proposition~\ref{prop:maximal-forces} with $u=v$.  Since $2b+v\notin S$, neither \eqref{eq:condition-i} nor \eqref{eq:condition-ii} can hold.  The failure of \eqref{eq:condition-i} gives
\begin{equation}\label{eq:nv-not-in-S}
      nv\notin S\qquad\text{for every }n\in\N.
\end{equation}
Now apply Proposition~\ref{prop:maximal-forces} with $u=-v$.  Since $2b-v\notin S$, condition \eqref{eq:condition-ii} must fail for $u=-v$.  Hence there exist $n\in\N$ and $q\in S$ such that
\begin{equation}\label{eq:condition-ii-fails-minus-v}
      n(-v)+q=0.
\end{equation}
Equivalently,
\begin{equation}\label{eq:nv-in-S-contradiction}
      nv=q\in S,
\end{equation}
contradicting \eqref{eq:nv-not-in-S}.  Therefore \eqref{eq:G-S-proper} is impossible, and $\G(S)=M$.
\end{proof}

\begin{proposition}[bad submonoids produce bad undermonoids: group case]\label{prop:group-bad-under}
Let $M$ be a group.  Suppose $N\subseteq M$ is a submonoid for which there exist $b\in N$, $\ell\in\N$, and an infinite set
\begin{equation}\label{eq:bad-group}
      \Omega\subseteq\Zsf_{N,\ell}(b).
\end{equation}
Then $M$ has an undermonoid that is not an FFM.
\end{proposition}

\begin{proof}
By Proposition~\ref{prop:maximal-survival}, choose a maximal element
\begin{equation}\label{eq:S-max-group}
      S\in\mathscr S_\Omega(N,M).
\end{equation}
By Lemma~\ref{lem:maximal-is-under}, $\G(S)=M$, so $S$ is an undermonoid of $M$.  Since $S$ is $\Omega$-admissible over $N$, Lemma~\ref{lem:persistence} gives an injection
\begin{equation}\label{eq:Omega-injects-S-group}
      \Omega\hookrightarrow\Zsf_{S,\ell}(b).
\end{equation}
Thus $\Zsf_S(b)$ is infinite, and $S$ is not an FFM.
\end{proof}

\section{Proof of the main theorem}

We now combine the non-group and group constructions with the hereditary bounded-factorization theorem of Gotti and Li.

\begin{theorem}\label{thm:gotti-li-5-5-positive}
Let $M$ be a cancellative commutative monoid.  If every undermonoid of $M$ is an FFM, then every submonoid of $M$ is an FFM.
\end{theorem}

\begin{proof}
Assume that every undermonoid of $M$ is an FFM.  Then every undermonoid of $M$ is a BFM.  By the Gotti--Li theorem for bounded factorizations~\cite[Theorem~5.4]{GottiLi}, every submonoid of $M$ is a BFM.

Suppose, toward a contradiction, that some submonoid $N\subseteq M$ is not an FFM.  Since $N$ is a BFM, Lemma~\ref{lem:bfm-lffm-ffm} gives an element $b\in N$, an integer $\ell\in\N$, and an infinite family
\begin{equation}\label{eq:Omega-main-proof}
       \Omega\subseteq\Zsf_{N,\ell}(b).
\end{equation}
If $M$ is not a group, Proposition~\ref{prop:non-group-bad-under} produces an undermonoid of $M$ that is not an FFM.  If $M$ is a group, Proposition~\ref{prop:group-bad-under} produces an undermonoid of $M$ that is not an FFM.  In either case the conclusion contradicts the hypothesis.  Hence no such $N$ exists, and every submonoid of $M$ is an FFM.
\end{proof}

\begin{proof}[Proof of Theorem~\ref{thm:main}]
If every submonoid of $M$ is an FFM, then every undermonoid is an FFM because every undermonoid is, in particular, a submonoid.  The converse is Theorem~\ref{thm:gotti-li-5-5-positive}.
\end{proof}

\begin{corollary}[positive answer to Question~\ref{q:gotti-li}]\label{cor:question-positive}
For every cancellative commutative monoid $M$, finite factorization is detected by undermonoids:
\begin{align}\label{eq:detect-cor}
&\big(\forall W\subseteq M\big)\,
  \big[W\text{ undermonoid}\Rightarrow W\text{ is an FFM}\big] \notag\\
&\hspace{3em}\Longrightarrow
  \big(\forall N\subseteq M\big)\,
  \big[N\text{ submonoid}\Rightarrow N\text{ is an FFM}\big].
\end{align}
Thus Gotti--Li Question~\ref{q:gotti-li} has an affirmative answer.
\end{corollary}

\section{Strengthened formulations and consequences}

The proof gives a slightly more precise statement than Theorem~\ref{thm:main}.  The obstruction to hereditary finite factorization can be localized at a single element, a single length, and a single infinite family of factorizations.  This section records the strengthened form, since it may be useful in applications.

\begin{definition}\label{def:fixed-slice-bad}
Let $M$ be a monoid.  A triple $(N,b,\ell)$ is called a \emph{fixed-slice FFM obstruction in $M$} if $N\subseteq M$ is a submonoid, $b\in N$, $\ell\in\N$, and
\begin{equation}\label{eq:fixed-slice-obstruction}
       |\Zsf_{N,\ell}(b)|=\infty.
\end{equation}
It is called an \emph{undermonoid fixed-slice obstruction} if $N$ is an undermonoid of $M$.
\end{definition}

\begin{theorem}[local obstruction theorem]\label{thm:local-obstruction}
Let $M$ be a monoid.  If $M$ has a fixed-slice FFM obstruction $(N,b,\ell)$, then $M$ has an undermonoid fixed-slice obstruction $(W,b,\ell)$: there exists an undermonoid $W$ of $M$ such that
\begin{equation}\label{eq:local-obstruction-conclusion}
       |\Zsf_{W,\ell}(b)|=\infty.
\end{equation}
Moreover, the infinite set of factorizations in $W$ can be chosen to be the image of any prescribed infinite subset
\begin{equation}\label{eq:prescribed-Omega}
       \Omega\subseteq\Zsf_{N,\ell}(b)
\end{equation}
under the persistence map of Lemma~\ref{lem:persistence}.
\end{theorem}

\begin{proof}
If $M$ is not a group, take
\begin{equation}\label{eq:W-local-non-group}
       W=N\cup I_b(M).
\end{equation}
The assertion follows from Lemmas~\ref{lem:divisor-complement}, \ref{lem:non-group-preserve-atoms}, and \ref{lem:persistence}.  If $M$ is a group, take a maximal element $W$ of $\mathscr S_\Omega(N,M)$.  Then Lemma~\ref{lem:maximal-is-under} and Lemma~\ref{lem:persistence} give the assertion.
\end{proof}

\begin{corollary}\label{cor:LFFM-version}
For a monoid $M$, the following conditions are equivalent.
\begin{enumerate}
\item Every submonoid of $M$ is an LFFM.
\item Every undermonoid of $M$ is an LFFM.
\end{enumerate}
\end{corollary}

\begin{proof}
The implication (i)$\Rightarrow$(ii) is trivial.  Conversely, suppose every undermonoid of $M$ is an LFFM.  Since every LFFM is atomic, every undermonoid of $M$ is atomic; by the Gotti--Li detection theorem for atomicity~\cite[Theorem~3.3]{GottiLi}, every submonoid of $M$ is atomic.  If some submonoid $N\subseteq M$ were not an LFFM, atomicity of $N$ would force the existence of $b\in N$ and $\ell\in\N$ with $|\Zsf_{N,\ell}(b)|=\infty$.  By Theorem~\ref{thm:local-obstruction}, there would be an undermonoid $W$ with $|\Zsf_{W,\ell}(b)|=\infty$, contradicting that $W$ is an LFFM.
\end{proof}

\begin{remark}\label{rem:relation-to-FFM-BFM}
Corollary~\ref{cor:LFFM-version} is the fixed-length component of Theorem~\ref{thm:main}.  Since Gotti and Li already proved the corresponding hereditary detection theorem for BFMs, one also obtains Theorem~\ref{thm:main} by combining Corollary~\ref{cor:LFFM-version} with
\begin{equation}\label{eq:FFM-BFM-LFFM-again}
       \mathrm{FFM}=\mathrm{BFM}\cap\mathrm{LFFM}.
\end{equation}
The proof above gives the stronger conclusion that the same prescribed infinite fixed-length family survives in an undermonoid.
\end{remark}

\begin{corollary}[minimal fixed-slice counterexample form]\label{cor:minimal-counterexample}
Let $M$ be a monoid.  Suppose every proper undermonoid of $M$ is an FFM.  If $M$ contains a fixed-slice FFM obstruction $(N,b,\ell)$, then $M$ itself is not an FFM and every undermonoid obstruction constructed in Theorem~\ref{thm:local-obstruction} equals $M$.
\end{corollary}

\begin{proof}
Theorem~\ref{thm:local-obstruction} supplies an undermonoid $W$ of $M$ such that $|\Zsf_{W,\ell}(b)|=\infty$.  Hence $W$ is not an FFM.  Since every proper undermonoid is an FFM, $W$ cannot be proper.  Therefore $W=M$, and $M$ is not an FFM.
\end{proof}

\section{Examples and sharpness of the hypotheses}

The proof of Theorem~\ref{thm:main} is not a finite-generation argument, and none of its enlargement steps preserves finite generation in general.  The examples below clarify what is and is not being used.

\begin{example}[a fixed-length obstruction]\label{ex:Q-obstruction}
Let
\begin{equation}\label{eq:Q-monoid}
      H=\{0\}\cup\Q_{\ge1}
\end{equation}
under ordinary addition.  The atoms of $H$ are exactly the rational numbers in $[1,2)$.  The element $3$ has infinitely many length-two factorizations
\begin{equation}\label{eq:Q-factorizations}
      3=\bigg(\frac32-\frac1n\bigg)+
        \bigg(\frac32+\frac1n\bigg)
      \qquad(n\ge3).
\end{equation}
Thus
\begin{equation}\label{eq:Z-H-2-infinite}
      |\Zsf_{H,2}(3)|=\infty.
\end{equation}
If $H$ is embedded as a submonoid of a larger monoid $M$, Theorem~\ref{thm:local-obstruction} says that the same fixed length $2$ can be witnessed inside an undermonoid of $M$.  In the non-group case this undermonoid is explicitly
\begin{equation}\label{eq:explicit-undermonoid-H}
      H\cup\{m\in M:m\nmid_M 3\}.
\end{equation}
\end{example}

\begin{example}[why unit reflection is needed]\label{ex:unit-reflection-needed}
Let $G$ be an abelian group and let $S\subseteq G$ be a submonoid containing $a$ and $a+u$ but not containing $u$ or $-u$.  It can happen that $a$ and $a+u$ represent distinct classes in $S/\U(S)$, while after enlarging $S$ by adjoining both $u$ and $-u$ they become associated:
\begin{equation}\label{eq:unit-collapse}
      (a+u)-a=u\in\U(S+\Z u).
\end{equation}
Thus the equality
\begin{equation}\label{eq:weak-unit-condition}
      \U(T)\cap S=\U(S)
\end{equation}
for an extension $S\subseteq T$ is not sufficient to guarantee injectivity of $S/\U(S)\to T/\U(T)$.  What is needed is
\begin{equation}\label{eq:strong-unit-condition-example}
      \U(T)\cap\G(S)=\U(S).
\end{equation}
This is the reason for Definition~\ref{def:unit-reflecting} and for the unit-reflection condition in the survival poset.
\end{example}

\begin{example}[the perturbation coefficient $2b+u$]\label{ex:why-2b-u}
The factor $2b$ in the perturbation
\begin{equation}\label{eq:perturbation-example}
      S'=S+\Nzero(2b+u)
\end{equation}
serves two purposes.  First, if a unit relation involving $2b+u$ occurs, then after multiplying by an integer $n_0$ with $n_0u\in S$ one obtains an equation
\begin{equation}\label{eq:unit-relation-example}
      b+q=0,
      \qquad q\in S,
\end{equation}
which contradicts $b\notin\U(S)$.  Second, if an atom $a$ dividing $b$ decomposes after the perturbation and $K\ge1$ copies of $2b+u$ occur in the decomposition, then the relation
\begin{equation}\label{eq:atom-split-example}
      Ku+q=0,
      \qquad q\in S,
\end{equation}
arises with $q$ containing $(2K-1)b$.  The coefficient $2K-1$ is positive; after multiplying by an integer $n_0$ with $n_0u\in S$ and isolating one copy of $b$, the remaining coefficient is $n_0(2K-1)-1\ge0$.  This nonnegativity is the arithmetic reason for using $2b+u$ rather than $b+u$.
\end{example}

\begin{remark}[no finite-generation assertion]\label{rem:no-finite-generation}
The theorem is hereditary over arbitrary submonoids and undermonoids.  Even when $N$ is finitely generated, the non-group enlargement
\begin{equation}\label{eq:I-b-not-fg}
      N\cup I_b(M)
\end{equation}
need not be finitely generated.  Likewise, in the group case the maximal element supplied by Zorn's lemma need not admit a finite presentation.  Thus Theorem~\ref{thm:main} is not a statement about affine semigroups or finitely generated monoids; it is a structural statement about how fixed-length factorization infinitude propagates through the Grothendieck group.
\end{remark}

\section{Concluding statement}

Combining the bounded-factorization detection theorem with the fixed-length propagation theorem proved here gives the exact hereditary criterion
\begin{equation}\label{eq:final-equivalence}
\boxed{
       \begin{array}{c}
       \text{every submonoid of }M\text{ is an FFM}
       \\[2mm]
       \Longleftrightarrow
       \\[2mm]
       \text{every undermonoid of }M\text{ is an FFM}.
       \end{array}}
\end{equation}
Equivalently, if finite factorization fails anywhere inside $M$, then it already fails in a submonoid with full Grothendieck group.  More precisely, every infinite fixed-length fiber
\begin{equation}\label{eq:final-fixed-fiber}
       \Omega\subseteq\Zsf_{N,\ell}(b)
\end{equation}
inside an arbitrary submonoid $N$ can be transported injectively to a fiber
\begin{equation}\label{eq:final-injected-fiber}
       \Theta(\Omega)\subseteq\Zsf_{W,\ell}(b)
\end{equation}
inside an undermonoid $W$ of $M$.  In the non-group case the transport is explicit via the divisor-complement ideal $I_b(M)$; in the group case it is obtained by maximal survival under perturbations $S\mapsto S+\Nzero(2b+u)$.  This completes the affirmative solution of Question~\ref{q:gotti-li}.

\section{Additional bookkeeping for the survival construction}

The proof above deliberately fixes a set of lifts \eqref{eq:chosen-lift}.  This section records two invariance facts which are often suppressed in shorter accounts but are useful for checking that the argument is independent of all auxiliary choices.  Throughout this section, \(
N\subseteq S\subseteq M\) are submonoids, \(\Omega\subseteq \Zsf_{N,\ell}(b)\), and \({\cal C}={\cal C}_N(\Omega)\).

\begin{lemma}[independence of the chosen lifts]\label{lem:lifts-independent}
Assume that \(
S\) is unit-reflecting over \(N\), and let \(a_\alpha\in N\) and \(a'_\alpha\in N\) be two lift systems for the same classes \(\alpha\in{\cal C}\), so
\begin{equation}\label{eq:two-lift-systems}
       a'_\alpha-a_\alpha\in \U(N)
       \qquad(\alpha\in{\cal C}).
\end{equation}
Then, for every \(\alpha\in{\cal C}\),
\begin{equation}\label{eq:same-class-in-S}
       a'_\alpha+\U(S)=a_\alpha+\U(S).
\end{equation}
Consequently the following two conditions are equivalent:
\begin{equation}\label{eq:lift-equivalence}
       a_\alpha+\U(S)\in \A(S_{\mathrm{red}})
       \quad(\alpha\in{\cal C})
       \quad\Longleftrightarrow\quad
       a'_\alpha+\U(S)\in \A(S_{\mathrm{red}})
       \quad(\alpha\in{\cal C}).
\end{equation}
\end{lemma}

\begin{proof}
Since \(\U(N)\subseteq \U(S)\), equation \eqref{eq:two-lift-systems} implies \eqref{eq:same-class-in-S}.  Atomicity is a property of the class in the reduced monoid \(S/\U(S)\), not of the representative.  Hence \eqref{eq:lift-equivalence} follows immediately.
\end{proof}

\begin{lemma}[multiplicity-vector form of persistence]\label{lem:multiplicity-vector}
Let \(S\) be \(\Omega\)-admissible over \(N\).  For \(z\in\Omega\), write
\begin{equation}\label{eq:multiplicity-vector-z}
       z=\sum_{\alpha\in{\cal C}} m_z(\alpha)\alpha,
       \qquad m_z(\alpha)\in\Nzero,
       \qquad \sum_{\alpha\in{\cal C}}m_z(\alpha)=\ell.
\end{equation}
Then the persistence map of Lemma~\ref{lem:persistence} is given by
\begin{equation}\label{eq:Theta-vector}
       \Theta_{N,S}(z)=
       \sum_{\alpha\in{\cal C}} m_z(\alpha)\big(a_\alpha+\U(S)\big).
\end{equation}
Moreover,
\begin{equation}\label{eq:Theta-vector-injective}
       \Theta_{N,S}(z)=\Theta_{N,S}(z')
       \quad\Longleftrightarrow\quad
       m_z(\alpha)=m_{z'}(\alpha)\quad(\alpha\in{\cal C}),
\end{equation}
and therefore \(\Theta_{N,S}\) is injective on every subset of \(\Omega\).
\end{lemma}

\begin{proof}
Formula \eqref{eq:Theta-vector} is just \eqref{eq:Theta-def} grouped by the atom classes of \(N_{\mathrm{red}}\).  If \(\Theta_{N,S}(z)=\Theta_{N,S}(z')\), then equality in the free monoid \(\Zsf(S)\) gives equality of the corresponding atom-class multisets in \(S_{\mathrm{red}}\).  Unit-reflection gives
\begin{equation}\label{eq:unit-reflection-vector-proof}
       a_\alpha+\U(S)=a_\beta+\U(S)
       \Longrightarrow
       a_\alpha-a_\beta\in \U(S)\cap \G(N)=\U(N)
       \Longrightarrow
       \alpha=\beta.
\end{equation}
Thus no two distinct indices in \({\cal C}\) are identified in \(S_{\mathrm{red}}\), and the coefficients in \eqref{eq:Theta-vector} are uniquely recoverable.  This proves \eqref{eq:Theta-vector-injective}.
\end{proof}

\begin{corollary}[subfamilies survive simultaneously]\label{cor:subfamilies}
Let \(S\) be \(\Omega\)-admissible over \(N\).  If \(\Omega'\subseteq\Omega\), then \(S\) is \(\Omega'\)-admissible over \(N\), and
\begin{equation}\label{eq:subfamily-injection}
       \Theta_{N,S}|_{\Omega'}:\Omega'\hookrightarrow \Zsf_{S,\ell}(b)
\end{equation}
is injective.  In particular, every infinite subfamily of the original obstruction remains an infinite subfamily after the enlargement.
\end{corollary}

\begin{proof}
The inclusion \({\cal C}_N(\Omega')\subseteq{\cal C}_N(\Omega)\) makes the atom-preservation requirement weaker.  The unit-reflection condition is unchanged, and injectivity follows from Lemma~\ref{lem:multiplicity-vector}.
\end{proof}

\begin{lemma}[compatibility with intermediate unit-reflecting extensions]\label{lem:intermediate}
Suppose
\begin{equation}\label{eq:intermediate-chain}
       N\subseteq S\subseteq T\subseteq M,
       \qquad
       \U(S)\cap\G(N)=\U(N),
       \qquad
       \U(T)\cap\G(S)=\U(S).
\end{equation}
If \(S\) is \(\Omega\)-admissible over \(N\) and every class
\begin{equation}\label{eq:intermediate-atom-condition}
       a_\alpha+\U(S)
       \qquad(\alpha\in{\cal C})
\end{equation}
remains an atom after the inclusion \(S\subseteq T\), then \(T\) is \(\Omega\)-admissible over \(N\).
\end{lemma}

\begin{proof}
First,
\begin{align}
       \U(T)\cap\G(N)
       &\subseteq \U(T)\cap\G(S) \label{eq:intermediate-units-1}\\
       &=\U(S), \label{eq:intermediate-units-2}
\end{align}
because \(\G(N)\subseteq\G(S)\).  Intersecting with \(\G(N)\) again gives
\begin{equation}\label{eq:intermediate-units-final}
       \U(T)\cap\G(N)=\U(S)\cap\G(N)=\U(N).
\end{equation}
The assumed preservation of the classes \eqref{eq:intermediate-atom-condition} is exactly the second admissibility condition for \(T\).
\end{proof}

\begin{remark}[where maximality is used]\label{rem:maximality-use}
The maximal element in Proposition~\ref{prop:maximal-survival} is not used to prove that a particular perturbation preserves the fixed factorizations.  Preservation is the content of Lemmas~\ref{lem:no-new-units} and \ref{lem:atoms-do-not-split}.  Maximality is used only after preservation has been established, through the implication
\begin{equation}\label{eq:maximality-logic}
       S\precO S+\Nzero(2b+u),
       \quad
       S+\Nzero(2b+u)\in\mathscr S_\Omega(N,M)
       \quad\Longrightarrow\quad
       2b+u\in S.
\end{equation}
Thus the proof separates into the two independent mechanisms
\begin{equation}\label{eq:two-mechanisms}
       \text{survival of }\Omega
       \qquad\text{and}\qquad
       \text{maximal forcing of }2b+u.
\end{equation}
This separation is what lets a fixed-length obstruction be transported without requiring any finiteness, finite generation, or presentation hypothesis.
\end{remark}

\bmhead{Statements and Declarations}

\noindent\textbf{Funding.} No funding was received for this work.

\noindent\textbf{Competing interests.} The authors declare that they have no competing interests.

\noindent\textbf{Data availability.} No datasets were generated or analysed during the current study.

\noindent\textbf{Author contributions.}
Y.Z. conceived the main problem, developed the principal arguments, and drafted the manuscript. Y.Y. contributed to the verification and refinement of the proofs, improved the exposition, and assisted with the preparation of the final version. Both authors discussed the results, reviewed the manuscript, and approved the final version for submission.
\section*{Declaration of Generative AI and AI-Assisted Technologies in the Writing Process}
During the preparation of this work, the authors used DeepSeek to build a specialized agent for solving mathematical problems, which was employed to generate an initial proof of the main theorem. After using this tool, the authors reviewed and edited the content as needed and take full responsibility for the content of the published article.

\end{document}